\documentclass[oneside,english]{amsart}
\usepackage[T1]{fontenc}
\usepackage[latin9]{inputenc}
\usepackage{amstext}
\usepackage{amsthm}
\usepackage{amssymb}

\makeatletter
\numberwithin{equation}{section}
\numberwithin{figure}{section}
\theoremstyle{plain}
\newtheorem{thm}{\protect\theoremname}[section]
\theoremstyle{plain}
\newtheorem{prop}[thm]{\protect\propositionname}
\theoremstyle{plain}
\newtheorem{lem}[thm]{\protect\lemmaname}
\theoremstyle{remark}
\newtheorem{rem}[thm]{\protect\remarkname}

\usepackage[top=30truemm,bottom=30truemm,left=25truemm,right=25truemm]{geometry}
\usepackage{graphics}

\newcommand{\stodd}{\reflectbox{$\ddots$}}

\allowdisplaybreaks[4]

\author{Nozomi Ito}
\makeatother

\usepackage{babel}
\providecommand{\lemmaname}{Lemma}
\providecommand{\propositionname}{Proposition}
\providecommand{\remarkname}{Remark}
\providecommand{\theoremname}{Theorem}

\begin{document}
\global\long\def\ad{\text{ad}}%
\global\long\def\abs{|\cdot|}%
\global\long\def\irr{\mathrm{Irr}}%
\global\long\def\spec{\mathrm{Spec}}%
\global\long\def\gll{\mathrm{GL}}%
\global\long\def\nat{\mathbb{\mathbb{Z}}_{>0}}%
\global\long\def\defi{\overset{\mathrm{def}}{\iff}}%
\global\long\def\map{\rightarrow}%
\global\long\def\ep{\varepsilon}%
\global\long\def\sll{\mathrm{SL}}%
\global\long\def\spp{\mathrm{Sp}}%
\global\long\def\diag{\mathrm{diag}}%
\global\long\def\tr{\mathrm{Tr}}%
\global\long\def\akk{\mathbb{A}_{k}}%
\global\long\def\aaa{\mathbb{A}}%
\global\long\def\aff{\mathbb{A}_{F}}%
\global\long\def\aee{\mathbb{A}_{E}}%
\global\long\def\akk{\mathbb{A}_{k}}%
\global\long\def\zz{\mathbb{Z}}%
\global\long\def\rr{\mathbb{R}}%
\global\long\def\qq{\mathbb{Q}}%
\global\long\def\cc{\mathbb{C}}%
\global\long\def\inj{\hookrightarrow}%
\global\long\def\surj{\twoheadrightarrow}%
\global\long\def\oo{\infty}%
\global\long\def\bu{\mathrm{U}}%
\global\long\def\sgn{\mathrm{\mathrm{sgn}}}%
\global\long\def\ind{\mathrm{\mathrm{Ind}}}%
\global\long\def\vol{\mathrm{\mathrm{vol}}}%
\global\long\def\aqq{\mathbb{A}_{\qq}}%
\global\long\def\nami{\rightsquigarrow}%
\global\long\def\lig{\mathfrak{g}}%
\global\long\def\lik{\mathfrak{k}}%
\global\long\def\lip{\mathfrak{p}}%
\global\long\def\lit{\mathfrak{t}}%
\global\long\def\liu{\mathfrak{u}}%
\global\long\def\lih{\mathfrak{h}}%
\global\long\def\liz{\mathfrak{Z}}%
\global\long\def\rint{\mathcal{O}}%
\global\long\def\soo{\mathrm{SO}}%
\global\long\def\act{\curvearrowright}%
\global\long\def\tran{^{t}}%

\title{The square-integrability of double descent}
\begin{abstract}
Double descent is a method to construct automorphic representations
of classical groups. For given A-parameter $\psi$ with certain good
properties, double descent constructs a space of functions orthogonal
to any cuspidal representation whose A-parameter is not $\psi$ and
not orthogonal to any cuspidal representation with A-parameter $\psi$.
In this paper, we show that functions constructed by double descent
are always square-integrable.
\end{abstract}

\maketitle

\section{Introduction\label{sec:Introduction}}

In \cite{zbMATH07485546}, Ginzburg and Soudry introduced double descent.
Double descent is a method to construct any irreducible cuspidal automorphic
representations of classical groups with given A-parameter with certain
good properties (we do not consider double descent for metaplectic
groups in this paper). For the case when the given A-parameter is
generic, double descent is completely established. However, for the
case when the given A-parameter is non-generic, it is incomplete due
to a certain difficulty. The purpose of this paper is to resolve the
difficulty.

Throughout this paper, let $F$ be a number field, and $G$ a symplectic
or special orthogonal group over $F$. We denote by $\aff$ the ring
of adeles of $F$ and by $F_{v}$ the completion of $F$ at any place
$v$ of $F$. Fix an appropriate maximal compact subgroup $K=\prod_{v}K_{v}$
of $G(\aff)$. Put $\lig={\rm Lie}(G(\rr\otimes_{\qq}F))\otimes_{\rr}\cc$
and $Z(\lig)={\rm Center}(U(\lig))$.

\subsection{Double descent}

For the sake of explanation, let us recall a special case of double
descent for odd special orthogonal group. Let $F$ be a number field.
For any $d\in\zz_{\geq0}$, denote by ${\rm SO}_{d}$ the split special
orthogonal group of rank $[d/2]$ over $F$. Let $m$ be an integer
and $\tau=\otimes_{v}\tau_{v}$ a cuspidal automorphic representation
of $\gll_{m}(\aff)$. For any $N\in\zz_{>0}$, we denote by $\Delta(\tau,N)=\otimes_{v}\Delta(\tau_{v},N)$
the isobaric automorphic representation 
\[
\tau|\det|_{\aff}^{(N-1)/2}\boxplus\tau|\det|_{\aff}^{(N-3)/2}\boxplus\dots\boxplus\tau|\det|_{\aff}^{-(N-1)/2}
\]
 of $\gll_{mN}(\aff)$ for short. Let $n$ be an odd integer and we
temporary put $G=\soo_{nm+1}$. Assume that $L(s,\tau,\wedge^{2})$
has a pole at $s=1$ (note that $m$ is automatically even). Put $\psi=\tau[n]$,
which is an elliptic A-parameter of $G$ (see $\S$\ref{sec:Preliminary}).

Put $H=\soo_{2m(nm+1)}$. Let $Q=MU$ be a parabolic subgroup of $H$
with Levi factor $M$ and unipotent radical $U$ such that $M\simeq(\gll_{nm+1})^{m-1}\times\soo_{2(nm+1)}$.
Let $\psi_{U}$ be a certain character of $U(F)\backslash U(\aff)$
such that the subgroup of $M(\aff)$ consists of elements stabilizing
$\psi_{U}$ is isomorphic to $G(\aff)\times G(\aff)$. For any measurable
and locally $L^{1}$ function $\phi$ on $U(F)\backslash H(\aff)$,
we define the Fourier coefficient $\mathcal{F}_{\psi_{U}}(\phi)$
of $\phi$ with respect to $\psi_{U}$ by
\[
\mathcal{F}_{\psi_{U}}(\phi)(g,h)=\int_{U(F)\backslash U(\aff)}\phi(u(g,h))\psi_{U}(u)du,
\]
where $g,h\in G(\aff)$ and $(g,h)$ is the image of $(g,h)\in G(\aff)\times G(\aff)$
in the stabilizer of $\psi_{U}$. Note that $\mathcal{F}_{\psi_{U}}(\phi)$
is a function on $(G(F)\times G(F))\backslash(G(\aff)\times G(\aff))$.
Let $P$ be the Siegel parabolic subgroup. For a standard section
$f_{s}$ of ${\rm Ind}_{P(\aff)}^{H(\aff)}\Delta(\tau,mn+1)|\det|_{\aff}^{s}$,
where ${\rm Ind}$ is normalized parabolic induction, let $E(f_{s})$
be the Einstein series associated with $f_{s}$. Note that $E(f_{s})$
admits meromorphic continuation and has a simple pole at $s=n/2$.
We put

\[
\mathcal{DD}=\{\mathcal{F}_{\psi_{U}}{\rm Res}_{s=n/2}(E(f_{s}))\ |\ f_{s}:\mbox{standard section of }{\rm Ind}_{P(\aff)}^{H(\aff)}\Delta(\tau,mn+1)|\det|_{\aff}^{s}\}.
\]
Ginzburg and Soudry proved the following theorem.
\begin{thm}[\cite{zbMATH07485546}]
\label{thm:}
\begin{enumerate}
\item For any nonzero function $\mathcal{E}$ in $\mathcal{DD}$, the $G(F_{v})\times G(F_{v})$-representation
generated by $\mathcal{E}$ is equivalent to 
\[
\pi_{\psi_{v}}\boxtimes\pi_{\psi_{v}}^{\vee}
\]
for almost all finite place $v$ of $F$, where $\pi_{\psi_{v}}$
is the unramified irreducible representation of $G(F_{v})$ corresponding
to $\Delta(\tau_{v},n)$ under unramified local Langlands' correspondence.
\item Let $\pi_{1},\pi_{2}$ be irreducible cuspidal automorphic representations
of $G(\aff)$. The trilinear form
\[
L_{\pi_{1},\pi_{2}}:(\mathcal{E},f_{1},f_{2})\mapsto\int_{[G]}\mathcal{E}(g,h)f_{1}(g)f_{2}(h)dgdh,\ \mathcal{E}\in\mathcal{DD},\ f_{i}\in\pi_{i}
\]
is nonzero if $\overline{\pi_{2}}=\pi_{1}$ and the A-parameter of
$\pi_{1}$ (and $\pi_{2}$) is $\psi$ (by (1), the converse also
holds).
\item If $n=1$, then 
\[
\mathcal{DD}=\bigoplus_{\pi}\pi\boxtimes\overline{\pi},
\]
where the summation runs over all $\pi$ with A-parameter $\psi(=\tau[1])$.
\end{enumerate}
\end{thm}

The proofs of (1) and (2) are relatively easy. (1) follows from the
computation of the twisted Jacquet module (with respect to $(\psi_{U})_{v}$)
of the unramified constituent of ${\rm Ind}_{P(F_{v})}^{H(F_{v})}\Delta(\tau,mn+1)|\det|_{F_{v}}^{n/2}$
at almost all place $v$ of $F$. (2) follows from the theory of twisted
doubling \cite{MR3989257}. On the other hand, the proof of (3) is
quite hard. By very long calculation, Ginzburg and Soudry proved that
all nontrivial constant terms of any function in $\mathcal{DD}$ are
vanishing when $n=1$. However, their strategy did not work sufficiently
for $n>1$.

\subsection{The main theorem}

The goal of this paper is to extend Theorem \ref{thm:} (3) for general
$n$ (and general $G$). For this goal, we show

\[
\mathcal{DD}\subset L^{2}(G(F)\times G(F)\backslash G(\aff)\times G(\aff)).
\]
If we can show it, we can immediately conclude that
\[
\mathcal{DD}\supset\bigoplus_{\pi}\pi\boxtimes\overline{\pi},
\]
where the summation runs over all \emph{cuspidal} $\pi$ with A-parameter
$\psi$. (In this paper, we do not consider residual representations
appear in $\mathcal{DD}$. Wait for twisted doubling to be extended
to non-cuspidal automorphic forms.)

From now on, $G$ is a general classical group again. The main result
of this paper is the following.
\begin{thm}
\label{thm:Assume-that} Let $\psi$ be an elliptic A-parameter of
$G$ and $\varphi$ a $\cc$-valued, smooth, and $K$-finite function
on $G(F)\backslash G(\aff)$. Assume that the $G(F_{v})$-representation
$\pi_{v}$ generated by $\varphi$ is irreducible and equal to $\pi_{\psi_{v}}$
(see $\S$\ref{sec:Preliminary}) for almost all finite places of
$F$. Moreover, assume that $\psi$ can be described as follows:
\[
\psi=\tau[n]\boxplus(\boxplus_{i=1}^{m}\tau_{i}[1]),
\]
where $\tau\neq\tau_{i}$ for any $i=1,\dots,k$ (it automatically
holds if $n$ is even). Then, $\varphi$ is square-integrable.
\end{thm}

This result guarantees the square-integrability of functions constructed
by double descent.

This paper is organized as follows. In $\S$\ref{sec:Preliminary},
we introduce some notations and recall some facts regarding the endoscopic
classification. In $\S$\ref{sec:Square-integrability}, we prove
Theorem \ref{thm:Assume-that}. We use Langlands' square-integrability
criterion to prove Theorem \ref{thm:Assume-that}. However, we do
not assume that $\varphi$ in Theorem \ref{thm:Assume-that} is an
automorphic form (in the narrow sense). Therefore, $\S$\ref{sec:-finiteness}
is devoted to showing that $\varphi$ is an automorphic form (Theorem
\ref{thm:Any--valued,-smooth,}).

\section{Preliminaries\label{sec:Preliminary}}

In this section, we introduce some additional notations and recall
some facts we use regarding the endoscopic classification.

\subsection{The group $G$}

For convenience, we fix a realization of $G$ as follows. Put $J_{m}=\begin{pmatrix} &  & 1\\
 & \stodd\\
1
\end{pmatrix}\in{\rm Mat}_{m\times m}$ for $m\in\zz_{\geq0}.$ Let $d_{0}\in\zz_{\geq0}$ and $H$ a symmetric
($G:$ special orthogonal group) or alternating ($G:$ symplectic
group) matrix in ${\rm Mat}_{d_{0}\times d_{0}}(F)$ such that the
form defined by $H$ is anisotropic. Note that $d_{0}=0$ and $H$
is the zero matrix if $G$ is a symplectic group. For $r\in\zz_{\geq0}$,
define $G_{H,r}$ by
\[
G_{H,r}(A)=\{g\in\gll_{d_{0}+2r}(A)\ |\ g\begin{pmatrix} &  & J_{r}\\
 & H\\
\ep J_{r}
\end{pmatrix}{}^{t}g=\begin{pmatrix} &  & J_{r}\\
 & H\\
\ep J_{r}
\end{pmatrix},\ \det g=1\}
\]
for any $F$-algebra $A$, where $\ep=-1$ if $G$ is a symplectic
group and $\ep=1$ if $G$ is a special orthogonal group. We always
assume that $G=G_{H,r}$ for some $H$ and $r$.

\subsection{Standard parabolic subgroups}

Let $r_{0}\in\zz_{\geq0},r_{i}\in\zz_{>0}\ (i=1,\dots,k)$ and put
$r=\sum_{i=0}^{k}r_{i},\ \mathbf{r}=(r_{0};r_{1}\dots,r_{k})$. We
denote by $M_{\mathbf{r}}$ the image of the following map
\[
\gll_{r_{1}}\times\dots\times\gll_{r_{k}}\times G_{H,r_{0}}\ni(g_{1},\dots,g_{k},g_{0})\mapsto\begin{pmatrix}g_{1}\\
 & \ddots\\
 &  & g_{k}\\
 &  &  & g_{0}\\
 &  &  &  & \ddots
\end{pmatrix}\in G_{H,r}
\]
and put
\[
U_{\mathbf{r}}=\{\begin{pmatrix}1_{r_{1}} & * & * & * & *\\
 & \ddots & * & * & *\\
 &  & 1_{r_{k}} & * & *\\
 &  &  & 1_{d_{0}+2r_{0}} & *\\
 &  &  &  & \ddots
\end{pmatrix}\in G_{H,r}\},\ P_{\mathbf{r}}=M_{\mathbf{r}}U_{\mathbf{r}}.
\]
When a parabolic subgroup $P$ of $G_{H,r}$ is equal to $P_{\mathbf{r}}$
for some $\{r_{i}\}_{i=0}^{k}$, we say that $P$ is a standard parabolic
subgroup.

\subsection{Parabolic induction}

Let $r_{0}\in\zz_{\geq0},r_{i}\in\zz_{>0}\ (i=1,\dots,k)$ and put
$r=\sum_{i=0}^{k}r_{i}$. Let $\tau_{i}=\otimes_{v}'\tau_{i,v}$ be
an irreducible cuspidal automorphic representation of $\gll_{r_{i}}(\aff)$
for $i=1,\dots,k$ and $\mu=\otimes_{v}'\mu_{v}$ an irreducible cuspidal
automorphic representation of $G_{H,r_{0}}(\aff)$. We denote by

\[
\tau_{1}\times\dots\times\tau_{k}\rtimes\mu
\]
the space of $\cc$-valued, smooth, and $K$-finite functions $f$
on $M_{\mathbf{r}}(F)U_{\mathbf{r}}(\aff)\backslash G_{H,r}(\aff)$
such that the function
\[
M(\aff)\ni m\mapsto f(mk)\in\cc
\]
is in $\rho_{\mathbf{r}}(\tau^{(1)}\boxtimes\dots\boxtimes\tau^{(k)}\boxtimes\mu)$
for any $k\in K$, where $\rho_{\mathbf{r}}$ is the square root of
the modulus character of $P_{\mathbf{r}}(\aff)$. Namely, $\tau_{1}\times\dots\times\tau_{k}\rtimes\mu$
is the normalized ($K$-finite) parabolically induced representation.
We also define the local counterpart in the similar manner. Then we
have
\[
\tau_{1}\times\dots\times\tau_{k}\rtimes\mu\simeq\otimes_{v}'(\tau_{1,v}\times\dots\times\tau_{k,v}\rtimes\mu_{v}).
\]

\subsection{Endoscopic classification}

We recall the endoscopic classification (see \cite{arthur,MR3338302,kmsw,MR4776199}
for more details). We first recall elliptic A-parameters of $G=G_{H,r}$.
We denote by $[N]$ the $N$-dimensional irreducible algebraic representation
of $\sll_{2}(\cc)$ for $N\in\zz_{>0}$. An elliptic A-parameter of
$G$ is the following formal (commutative) sum
\[
\psi=\boxplus_{i}\tau_{i}[N_{i}],
\]
where
\begin{itemize}
\item $\tau_{i}[N_{i}]=\tau_{i}\boxtimes[N_{i}]$ for short,
\item $N_{i}$ is a positive integer,
\item $\tau_{i}$ is an irreducible cuspidal automorphic representation
of $\gll_{n_{i}}(\aff)$ such that 
\[
\sum_{i}n_{i}N_{i}=N:=\begin{cases}
2r+1 & G:\mbox{symplectic group},\\
2r+2[\frac{d_{0}}{2}] & G:\mbox{special orthogonal group},
\end{cases}
\]
\item if $(-1)^{\kappa+N+N_{i}}=1$ (resp. $(-1)^{\kappa+N+N_{i}}=-1$)
for
\[
\kappa=\begin{cases}
1 & G:\mbox{even special orthogonal group},\\
0 & \text{otherwise},
\end{cases}
\]
then the symmetric square L-function $L(s,\tau_{i},{\rm Sym}^{2})$
(resp. the exterior square L-function $L(s,\tau_{i},\wedge^{2})$)
has pole at $s=1$,
\item if $i\neq j$ and $N_{i}=N_{j}$, then $\tau_{i}\neq\tau_{j}$, and
\item if $G$ is a symplectic group or even special orthogonal group, then
the central character $\chi_{\tau_{i}}$ of $\tau_{i}$ satisfies
that $\prod_{i}\chi_{\tau_{i}}^{N_{i}}=\begin{cases}
1 & G:\mbox{symplectic group,}\\
\chi_{F(\delta)/F} & G:\mbox{even special orthogonal group,}\ \det H=-\delta
\end{cases}$, where $\chi_{F(\sqrt{\delta})/F}$ is the quadratic (or trivial)
character of $\aff^{\times}/F^{\times}$ corresponding to $F(\sqrt{\delta})/F$.
\end{itemize}
For the above $\psi$, define the associated isobaric representation
$\phi_{\psi}$ of $\psi$ by
\[
\phi_{\psi}=\boxplus_{i}(\tau_{i}|\det|_{\aff}^{(N_{i}-1)/2}\boxplus\tau_{i}|\det|_{\aff}^{(N_{i}-3)/2}\boxplus\dots\boxplus\tau_{i}|\det|_{\aff}^{-(N_{i}-1)/2}),
\]
which is an irreducible representation of $\gll_{N}(\aff)$. If $(\phi_{\psi})_{v}$
is irreducible unramified representation, then we denote by $\pi_{\psi_{v}}$
the irreducible unramified representation of $G(F_{v})$ corresponding
to $(\phi_{\psi})_{v}$ under unramified local Langlands' correspondence.
Then a crude version of the endoscopic classification is as follows.
\begin{thm}
\label{thm:The-discrete-spectrum}The discrete spectrum $L_{{\rm disc}}^{2}(G(F)\backslash G(\aff))$
of $G(F)\backslash G(\aff)$ has the following decomposition
\[
L_{{\rm disc}}^{2}(G(F)\backslash G(\aff))=\hat{\bigoplus_{\psi}}L_{{\rm \psi}}^{2}(G(F)\backslash G(\aff)),
\]
where the summation runs over all elliptic A-parameters $\psi$ of
$G$ and $L_{{\rm \psi}}^{2}(G(F)\backslash G(\aff))$ is the sum
of all irreducible discrete automorphic representation $\pi=\otimes_{v}\pi_{v}$
of $G(\aff)$ such that $\pi_{v}$ is equivalent to $\pi_{\psi_{v}}$
for almost all places $v$ of $F$.
\end{thm}

The following fact is important in the next section (recall that $\lig={\rm Lie}(G(\rr\otimes_{\qq}F))\otimes_{\rr}\cc$
and $Z(\lig)={\rm Center}(U(\lig))$).
\begin{prop}
\label{prop:Let--be-1}Let $\psi$ be an elliptic A-parameter of $G$.
Then, $L_{{\rm \psi}}^{2}(G(F)\backslash G(\aff))$ is decomposed
into finitely many $Z(\lig)$-eigensubspaces.
\end{prop}

This is an immediate consequence of the finiteness of an archimedean
local A-packet.

\section{$Z(\protect\lig)$-finiteness\label{sec:-finiteness}}

The aim of this section is to show the following.
\begin{thm}
\label{thm:Any--valued,-smooth,}Let $\varphi$ be a $\cc$-valued,
smooth, and $K$-finite function on $G(F)\backslash G(\aff)$. Assume
that $\varphi$ is $\otimes'_{v\in S}C_{c}^{\infty}(K_{v}\backslash G(F_{v})/K_{v})$-finite
for some cofinite subset $S$ of the set of finite places of $F$,
Then, $\varphi$ is $Z(\lig)$-finite.
\end{thm}

According to Averbuch (\cite{MR850116}), a $\cc$-valued, smooth,
and $K$-finite function on $G(F)\backslash G(\aff)$ is of uniform
moderate growth if it is $C_{c}^{\infty}(K_{v}\backslash G(F_{v})/K_{v})$-finite
for some finite place of $F$ such that the residue characteristic
of $F_{v}$ is sufficiently large. Therefore, we know that $\varphi$
in Theorem \ref{thm:Assume-that} is an automorphic form i.e. a $\cc$-valued,
smooth, $K$-finite, and $Z(\lig)$-finite function of uniform moderate
growth.

\subsection{$\protect\liz_{v}$-finite function on $M(F)U(\protect\aff)\backslash G(\protect\aff)$}

Let $v$ be a finite place of $F$ and denote by $\liz_{v}$ the Bernstein
center of $G(F_{v})$. Fix a prime element $\varpi_{v}$ of $F_{v}$.
Let $P=MU$ be a parabolic subgroup of $G$. Let $T\simeq\gll_{1}^{k}$
be the maximal split torus of the center of $M$. We denote by $\chi_{1},\dots,\chi_{k}$
the $\zz$-basis of ${\rm Rat}(M)$. Any $\lambda=\sum_{i}s_{i}\otimes\chi_{i}\in{\rm Rat}(M)\otimes_{\zz}\cc$
is regarded as a function on $M(F)\backslash M(\aff)$ by $\lambda(m)=\prod_{i}|\chi_{i}(m)|_{\aff}^{s_{i}}$.
Denote by $\rho_{P}\in{\rm Rat}(M)\otimes_{\zz}\frac{1}{2}\zz$ the
square root of the modulus character of $P(\aff)$. Denote by $A_{M}^{v}$
the image of $(\varpi_{v}^{\zz})^{k}\subset(\aff^{\times})^{k}$ in
$T(\aff)$. Then we can find the following fact in \cite[\S I.3.2]{MR1361168}.
\begin{lem}
\label{lem:Let--be}Let $f$ be a $\cc$-valued, smooth, $K$-finite
and $\liz_{v}$-finite function on $M(F)U(\aff)\backslash G(\aff)$.
Then, we have

\[
f=\sum_{i=1}^{n}Q_{i}\lambda_{i}f_{i},
\]
where
\begin{itemize}
\item $Q_{i}$ is a polynomial in $\cc[X_{1},\dots,X_{k}]$ $(i=1,\dots,n)$
and it is regarded as a function on $G(\aff)$ by 
\[
Q_{i}(g)=Q_{i}(\log|\chi_{1}(m)|_{\aff},\dots,\log|\chi_{k}(m)|_{\aff})
\]
 for $g=muk$ $(m\in M(\aff),u\in U(\aff),k\in K)$,
\item $\lambda_{i}\in{\rm Rat}(M)\otimes_{\zz}\cc$ $(i=1,\dots,n)$ and
it is regarded as a function on $G(\aff)$ by
\[
\chi_{i}(g)=\chi_{i}(m)
\]
 for $g=muk$ $(m\in M(\aff),u\in U(\aff),k\in K)$, and
\item $f_{i}$ is a function on $A_{M}^{v}M(F)U(\aff)\backslash G(\aff)$
such that $z_{i}f=\lambda_{i}f_{i}$ for some element $z_{i}$ of
the Bernstein center of $A_{M}^{v}$ ($z_{i}$ acts $f_{i}$ by left
translation).
\end{itemize}
\end{lem}

Note that $F^{\times}\varpi_{v}^{\zz}\backslash\aff^{\times}$ is
of finite volume. Thus, $A_{M}^{v}M(F)\backslash M(\aff)$ is of finite
volume. We decompose the cuspidal spectrum $L_{{\rm cusp}}^{2}(A_{M}^{v}M(F)\backslash M(\aff))$
as follows:

\[
L_{{\rm cusp}}^{2}(A_{M}^{v}M(F)\backslash M(\aff))=\hat{\bigoplus_{i}}\sigma_{i},
\]
where $\sigma_{i}\ (i=1,2\dots)$ is an irreducible unitary cuspidal
automorphic representation of $M(\aff)$ on which $A_{M}^{v}$ acts
trivially. Then we have the following.
\begin{lem}
\label{lem:Let--be-1}Let $f$ be a $\cc$-valued, smooth, $K$-finite
and cuspidal function on $M(F)U(\aff)\backslash G(\aff)$ of uniform
moderate growth. Then, we have

\[
f=\sum_{i=1}^{\infty}f_{i},\ f_{i}\in{\rm Ind}_{P(\aff)}^{G(\aff)}\rho_{P}^{-1}\sigma_{i},
\]
where ${\rm Ind}$ is normalized parabolic induction.
\end{lem}

\begin{proof}
Since $f$ is of uniformly moderate growth, $f$ is rapidly decreasing
(see \cite[\S I.2.12]{MR1361168}). Thus this lemma immediately follows
from the decomposition of $L_{{\rm cusp}}^{2}(A_{M}^{v}M(F)\backslash M(\aff))$.
\end{proof}
Note that we have not used the fact that $G$ is a classical group
yet.

\subsection{$\otimes'_{v\in S}C_{c}^{\infty}(K_{v}\backslash G(F_{v})/K_{v})$-finite
function on $M(F)U(\protect\aff)\backslash G(\protect\aff)$}
\begin{prop}
\label{prop:Let--be}Let $f$ be a $\cc$-valued, smooth, $K$-finite
and cuspidal function on $M(F)U(\aff)\backslash G(\aff)$. Assume
that $f$ is $\otimes'_{v\in S}C_{c}^{\infty}(K_{v}\backslash G(F_{v})/K_{v})$-finite
for a cofinite subset $S$ of the set of finite places of $F$. Then,
$f$ is $Z(\lig)$-finite.
\end{prop}

\begin{proof}
Assume that $v$ is in $S$ and describe
\[
f=\sum_{i=1}^{n}Q_{i}\lambda_{i}f_{i}
\]
as in Lemma \ref{lem:Let--be}. Note that $\lambda_{i}f_{i}=z_{i}f$
is a $\cc$-valued, smooth, $K$-finite, $\otimes'_{v\in S}C_{c}^{\infty}(K_{v}\backslash G(F_{v})/K_{v})$-finite,
and cuspidal function on $M(F)U(\aff)\backslash G(\aff)$ of uniform
moderate growth (by \cite{MR850116}). To prove this proposition,
It is sufficient to find an ideal $I_{i}$ of $Z(\lig)$ for each
$i=1,\dots,n$ such that $I_{i}\lambda_{i}f_{i}=0$ and $Z(\lig)/I_{i}$
is finite dimensional over $\cc$. If we can find them, $I=I_{1}^{M_{1}}\dots I_{n}^{M_{n}}$
satisfies that $Z(\lig)/I$ is finite dimensional over $\cc$ and
$If=0$ if $M_{i}\in\zz_{>0}$ are sufficiently large.

Put $f_{i}=\varphi$ and $\lambda_{i}=\lambda$ for short. For simplicity
we assume that $P$ is a standard parabolic subgroup. Then, by Lemma
\ref{lem:Let--be-1}, we have
\[
\lambda\varphi=\sum_{i=1}^{\infty}\varphi_{i},\ \varphi_{i}\in\tau_{i,1}\times\dots\times\tau_{i,k}\rtimes\mu_{i},
\]
where $\lambda^{-1}\rho_{P}(\tau_{i,1}\boxtimes\dots\boxtimes\tau_{i,k}\boxtimes\mu_{i})\ (i=1,2\dots)$
is an irreducible unitary cuspidal automorphic representation of $M(\aff)$
on which $A_{M}^{v}$ acts trivial. Denote the elliptic A-parameter
of $\mu_{i}$ by $\psi_{i}$. Then, since $\lambda\varphi$ is $\otimes'_{v\in S}C_{c}^{\infty}(K_{v}\backslash G(F_{v})/K_{v})$-finite,
there is a finite set of isobaric representation $\Phi$ such that
if $\varphi_{i}\neq0$, then $\phi_{\psi_{i}}\boxplus_{j=1}^{k}(\tau_{i,j}\boxplus\tau_{i,j}^{\vee})\in\Phi$,
where $\tau_{i,j}^{\vee}$ is the dual of $\tau_{i,j}$. Thus, there
is an integer $k>0$ such that if $\varphi_{i}\neq0$, then $\tau_{i,j}=\tau_{i',j}\ (j=1,\dots,k),\ \psi_{i}=\psi_{i'}$
for some $i'\in\{1,\dots,k\}$. Therefore, by Proportion \ref{prop:Let--be-1},
\[
\lambda\varphi=\sum_{i=1}^{\infty}\varphi_{i}=\sum_{i'=1}^{k}\sum_{i=i'}\varphi_{i}
\]
is $Z(\lig)$-finite, or in other words, there is an ideal $I_{i}$
of $Z(\lig)$ what we want.
\end{proof}

\subsection{Proof of Theorem \ref{thm:Any--valued,-smooth,}}

We are now ready to prove Theorem \ref{thm:Any--valued,-smooth,}. 
\begin{proof}[Proof of Theorem \ref{thm:Any--valued,-smooth,}]
Let $N$ be the number of standard parabolic subgroups $Q$ such
that the constant term $\varphi_{Q}$ of $\varphi$ with respect to
$Q$ is nonzero. Let $P$ be a standard parabolic subgroup such that
$\varphi_{P}\neq0$ and $\varphi_{Q}=0$ for any $Q\subset P$. Since
$\varphi_{P}$ is cuspidal, we can find an ideal $I$ of $Z(\lig)$
such that $I\varphi_{P}=0$ and the $\cc$-algebra $Z(\lig)/I$ is
a finite dimensional $\cc$-vector space by Proposition \ref{prop:Let--be}.
Then we have
\[
Z(\lig)\varphi=\sum_{i}\cc x_{i}\varphi+\sum_{j}Z(\lig)y_{j}\varphi,
\]
where $\{x_{i}\}_{i}\subset Z(\lig)$ such that $\{x_{i}+I\}_{i}$
is a basis of $Z(\lig)/I$ and $\{y_{j}\}_{j}\subset Z(\lig)$ is
a (finite) set of generator of $I$. Then, for any $j$, the number
of standard parabolic subgroups $Q$ such that $(y_{j}\varphi)_{Q}\neq0$
is smaller than $N$ since $(y_{j}\varphi)_{P}\in I\varphi_{P}=0$.
Thus, by induction, we conclude that $Z(\lig)f$ is finite-dimensional.
\end{proof}

\section{Square-integrability\label{sec:Square-integrability}}

In this section, we prove Theorem \ref{thm:Assume-that}. We assume
that $G=G_{H,r}$ for some $H$ and $r>0$ (if $r=0$, there is nothing
to show).

\subsection{Square-integrability criterion}

The important fact is that $\varphi$ in Theorem \ref{thm:Assume-that}
is an automorphic form by Theorem \ref{thm:Any--valued,-smooth,}.
Therefore, we can use Langlands' square-integrability criterion. Let
us recall it.
\begin{prop}[{\cite[\S 1.4.11.]{MR1361168}}]
\label{prop:Let--be-2}Let $\varphi$ be an automorphic form and
$\pi$ the $G(\aff)$-representation generated by $\varphi$. Consider
the following space
\[
{\rm Hom}_{G(\aff)}(\pi,\tau_{1}'|\det|_{\aff}^{R_{1}}\times\dots\times\tau_{k}'|\det|_{\aff}^{R_{k}}\rtimes\mu),
\]
where
\begin{itemize}
\item $\tau_{1}'\boxtimes\dots\boxtimes\tau_{k}'\boxtimes\mu$ is an irreducible
unitary cuspidal representation of $M_{(r_{0};r_{1},\dots,r_{k})}(\aff)$
and
\item $R_{i}\in\rr$.
\end{itemize}
If $\varphi$ is not square-integrable, then the above space is nonzero
for some $\{r_{i}\}_{i=0}^{k},\{\tau_{i}'\}_{i=1}^{k},\mu,$ and $\{R_{i}\}_{i=1}^{k}$
such that
\[
\sum_{i=1}^{l}r_{i}R_{i}\geq0
\]
for some $l\in\{1,\dots,k\}$.
\end{prop}

This is slightly weaker than the original version, but easier to handle
in terms of representation theory.

\subsection{Real number $d(-)$}

Next we introduce $d(-)$. Let $v$ be a finite place of $F$. For
any irreducible unramified representation $\sigma$ of $G(F_{v})$,
consider the following decomposition
\[
[J_{P_{(r-1;1)}}\sigma]=\bigoplus_{i=1}^{k}\abs_{F_{v}}^{s_{i}}\boxtimes\Sigma_{i},
\]
where
\begin{itemize}
\item $J_{P_{(r-1;,1)}}\sigma$ is the Jacquet module of $\sigma$ with
respect to $P_{(r-1;1)}(F_{v})$ and $[J_{P_{(r-1;1)}}\sigma]$ is
the semisimplification of $J_{P_{(r-1;1)}}\sigma$,
\item $s_{i}\in\cc$, and
\item $\Sigma_{i}$ is some smooth representation of $G_{H,r-1}(F_{v})$.
\end{itemize}
Then, we define $d(\sigma):=\max\{{\rm Re}s_{i}\ |\ i=1,\dots,k\}$.

For local components of discrete automorphic representations, $d(-)$
has the following property.
\begin{prop}
\label{prop:Let--be-3}Let $\sigma=\otimes_{v}\sigma_{v}$ be an irreducible
discrete automorphic representation of $G(\aff)$. Let $v$ be a finite
place of $F$. Assume that $\sigma_{v}$ is unramified and $G$ is
split over $F_{v}$. Then we have $d(\sigma_{v})<1/2$.
\end{prop}

\begin{proof}
Since $\sigma_{v}$ corresponds to the $v$-th component of associated
isobaric representation for some elliptic A-arameter, $\sigma_{v}$
is realized as the unramified subquotient of the following induced
representation

\[
|\det{}_{n_{1}}|_{F_{v}}^{t_{1}}\times\dots\times|\det{}_{n_{k}}|_{F_{v}}^{t_{k}}\rtimes\mu
\]
for some $n_{i}\in\zz_{>0}$, $t_{i}\in\cc$ such that $0\leq{\rm Re}(t_{i})<1/2$
and ``strongly negative unramified representation'' $\mu$ (see
\cite[Theorem 0-2,0-5]{MR2767523}). Note that $d(\mu)<0$. Thus,
it is easy to check that $d(\sigma)<1/2$ (e.g. \cite[Theorem 3.1.]{MR2231392}).
\end{proof}
\begin{rem}
\begin{itemize}
\item If Ramanujan conjecture is true, then we have $d(\sigma_{v})\leq0$.
\item Even if we remove ``and $G$ is split over $F_{v}$'', this theorem
would still hold. However, since \cite[Theorem 0-2,0-5]{MR2767523}
does not cover such cases, we cannot remove it.
\end{itemize}
\end{rem}

\subsection{Proof of Theorem \ref{thm:Assume-that}}

We are now ready to prove Theorem \ref{thm:Assume-that}.
\begin{proof}[Proof of Theorem \ref{thm:Assume-that}]
Let 
\[
\pi\map\tau_{1}'|\det|_{\aff}^{R_{1}}\times\dots\times\tau_{k}'|\det|_{\aff}^{R_{k}}\rtimes\mu,
\]
be a nonzero $G(\aff)$-map, where $\{r_{i}\}_{i=0}^{k},\{\tau_{i}'\}_{i=1}^{k},\mu,$
and $\{R_{i}\}_{i=1}^{k}$ are as in Proposition \ref{prop:Let--be-2}.
Let $\psi'$ be the A-parameter of $\mu$. Then, there is a nonzero
map

\[
\pi_{v}\map\tau_{1,v}'|\det|_{F_{v}}^{R_{1}}\times\dots\times\tau_{k,v}'|\det|_{F_{v}}^{R_{k}}\rtimes\mu_{v}
\]
for any $v\in S$ and we have
\[
(\tau|\det|_{\aff}^{(n-1)/2}\boxplus\tau|\det|_{\aff}^{(n-3)/2}\boxplus\dots\boxplus\tau|\det|_{\aff}^{-(n-1)/2})\boxplus(\boxplus_{i=1}^{m}\tau_{i})=\boxplus_{j=1}^{k}(\tau_{j}'|\det|_{\aff}^{R_{i}}\boxplus\tau_{j}'^{\vee}|\det|_{\aff}^{-R_{i}})\boxplus\phi_{\psi'}.
\]
Therefore, we have 
\begin{itemize}
\item $\psi'=\tau[n-2k]\boxplus(\boxplus_{i=1}^{m}\tau_{i}[1])$, 
\item $\tau_{i}'=\tau$ for $i=1,\dots,k$, and 
\item $\{|R_{i}|\}_{i=1}^{k}=\{(n-1)/2,(n-3)/2,\dots.(n-2k+1)/2\}$. 
\end{itemize}
We show $R_{i}<0$ for any $i$.

Assume that $R_{1},R_{2},\dots,R_{t-1}<0,\ R_{t}\geq0$ for some $t.$
Then $R_{t}\geq1/2$ and $R_{t}-R_{i}\geq2$ for $i=1,2,\dots,t-1$.
Let $v$ be a place of $F$ such that $v\in S$ and $G$ is split
in $F_{v}$. Then, we have

\[
\tau_{1,v}'|\det|_{F_{v}}^{R_{1}}\times\dots\times\tau_{k,v}'|\det|_{F_{v}}^{R_{k}}\rtimes\mu_{v}\simeq\tau_{t,v}'|\det|_{F_{v}}^{R_{t}}\times\tau_{1,v}'|\det|_{F_{v}}^{R_{t-1}}\times\dots\times\tau_{t-1,v}'|\det|_{F_{v}}^{R_{t+1}}\times\tau_{t+1,v}'|\det|_{F_{v}}^{R_{k}}\times\dots\times\tau_{k,v}'|\det|_{F_{v}}^{R_{k}}\rtimes\mu_{v}.
\]
Thus we have $d(\pi_{v})\geq1/2$. However, this contradicts to Proposition
\ref{prop:Let--be-3}.

Therefore, by Proposition \ref{prop:Let--be-2}, $\varphi$ is square-integrable.
\end{proof}
\bibliographystyle{plain}
\bibliography{cloudbib}

\end{document}